\documentclass{article}

\usepackage{amsmath,amssymb,url}
\usepackage{enumitem} 
\usepackage{framed, xcolor, nicefrac}
\usepackage{amsopn, amsthm}
\usepackage{tikz-cd}
\usetikzlibrary{graphs, graphs.standard, shapes.geometric, decorations.pathreplacing}

\numberwithin{equation}{section}

\theoremstyle{plain}
\newtheorem{theorem}{Theorem}[section]
\newtheorem{lemma}[theorem]{Lemma}
\newtheorem{claim}[theorem]{Claim}

\theoremstyle{definition}

\renewenvironment{proof}[1][\unskip]{%
\par
\noindent
\textbf{Proof #1.}
\noindent}
{\hfill$\blacksquare$

\bigskip}

\usepackage{etoolbox}

\newcommand{\defqedsymbol}{\hfill\ensuremath{\square}} 

\AtBeginEnvironment{definition}{\pushQED{\defqedsymbol}}
\AtEndEnvironment{definition}{\popQED}

\let\<=\langle
\let\>=\rangle
\def\land{\wedge}       
\def\lor{\vee}

\def\lnot{\neg}

\let\models=\vDash

\def\forces{\Vdash}
\def\notforces{\nVdash}
\renewcommand{\setminus}{\smallsetminus}

\usepackage[T1]{fontenc}
\usepackage[utf8]{inputenc}

\reversemarginpar
\usepackage[colorinlistoftodos]{todonotes}

\def\A{\mathfrak{A}}
\def\B{\mathfrak{B}}
\def\F{\mathfrak{F}}
\def\V{\mathbf{V}}
\def\H{\mathbf{H}}
\def\S{\mathbf{S}}
\def\P{\mathbf{P}}

\def\K{\mathsf{K}}
\def\E{\mathbf{E}}

\def\Sg{\mathsf{Sg}}
\def\Con{\mathsf{Con}}
\def\Fr{\mathbf{F}}

\begin{document}


\title{Uncountably many maximally consistent neighborhood complete congruential modal logics}

\author{Zal\'an Gyenis\thanks{Jagiellonian University}\and Krzysztof Krawczyk\thanks{Jagiellonian University}}
\date{\today}

\maketitle

\begin{abstract}
	We solve an open problem posed by Peter Fritz in \cite{Fritz} by proving that there are uncountably many C-Post complete congruential modal logics which are neighborhood complete. The method is algebraic: we construct an uncountable sequence of varieties of modal algebras which are minimal in the lattice of all subvarieties of modal algebras and are generated by a single complete and atomic modal algebra. 
\end{abstract}




\section{The problem}
David Makinson isolated four modal logics, each being semantically complete with respect to a single two element algebra, which exhaust all Post complete extensions in various lattices of modal logics \cite{Makinson}. In particular, two of them form the only Post complete extensions in the lattice of normal modal logics. The question whether this result extends to the lattice of congruential modal logics has been resolved negatively by Humberstone \cite{Humberstone}. He provided two examples of two-element neighborhood frames whose theories are not included in any of the four Makinson's logics. Since neighborhood frames always validate the congruentiality rule, there must be more Post complete extensions in the lattice of congruential modal logics.

In 2016, Peter Fritz has shown that the number of C-Post complete congruential modal logics in fact equals continuum \cite[Theorem 4.1]{Fritz}. Nonetheless, his proof is inconclusive with respect to the question whether the logics from his construction are neighborhood complete.  Therefore, he poses two open problems \cite[p. 293]{Fritz}
\begin{quote}
{\it [...]this raises the question how widespread
such incompleteness is among logics Post complete in the set of congruential
modal logics. The following result gives a partial answer, by showing that
there are infinitely many C-Post complete logics which are the logic of a class
of neighborhood frames. What their precise number is will be left open, as well
as the question whether there are any C-Post complete logics which fail to be
the logic of a class of neighborhood frames, and if so, how many such logics
there are.}
\end{quote}
He constructed a countable sequence of  C-Post complete congruential modal logics, each being a theory of a single finite modal matrix \cite[Theorem 4.2]{Fritz}. Any finite modal matrix/algebra has a modally equivalent neighborhood frame, therefore his construction yields the existence of countably infinitely many neighborhood complete, C-post complete congruential modal logics.

The second question has been recently solved in \cite{Krawczyk2026} where it was proven that there are uncountably many C-Post complete congruential modal logics that are not only neighborhood incomplete but also neighborhood unsound (strongly incomplete). In the current paper, we solve the first problem by showing that there are uncountably many C-Post complete and neighborhood complete congruential modal logics.

We will rely on algebraic techniques, thus finding Post-complete congruential modal logics amounts to finding minimal varieties of modal algebras.
Any minimal variety is generated by a zero-generated simple algebra (see Lemma \ref{zero-generated} below). But such algebras are countable. What is more, any countable algebra which is based on a neighborhood frame (i.e. atomic and complete) is finite. Therefore, there are only countably many varieties generated by a single algebra of this form. This poses a natural limitation in establishing the exact number of Post-complete neighborhood complete modal logics. If this number turns out to be continuum (which seems safe to assume) most of these logics cannot be identified with the theories of a single finite neighborhood frame. Therefore, we need to somehow maneuver around this obstacle.

Our strategy involves a construction based on a specifically designed continuous mappings from a Cantor's space into itself. For each subset $S$ of natural numbers we will construct a different continuous mapping. The Cantor space  (understood as a powerset of natural numbers) together with a unary operation given by one of such continuous mappings will comprise a modal algebra with an equivalent neighborhood frame. This will give us uncountably many neighborhood complete congruential modal logics. To show C-Post completeness, we will prove that each of such defined algebras has a modally equivalent 0-generated simple modal subalgebra.

\paragraph{Preliminaries.}
The standard propositional modal language is given by a countably infinite set of propositional letters and the usual logical connectives $\land$, $\lnot$, $\bot$, and $\Box$. Other connectives such as $\top$, $\to$, $\lor$, etc. are taken as defined in the familiar ways.
A set $\Lambda$ of formulas is a \emph{modal logic} if it contains all propositional tautologies, and is closed under modus ponens and uniform substitution. The modal logic $\Lambda$ is \emph{congruential} when additionally it is closed under the congruentiality rule: if $p\leftrightarrow q\in \Lambda$, then $\Box p\leftrightarrow \Box q\in
\Lambda$. The least congruential modal logic is denoted by $\E$.

A neighborhood frame $\F$ is a tuple $\<W, N\>$ such that $W$ is a nonempty set and $N:W\to \wp(\wp(W))$ is a function. A function $V$ assigning subsets of $W$ to propositional variables is called an evaluation, and an evaluation and a frame gives rise to a recursive definition of truth of formulas: for $w\in W$ define
\begin{align*}
\F,V,w&\forces p &\text{ iff }& w\in V(p)\qquad \text{ for propositional variable } p \\
\F,V,w&\forces \varphi\land\psi &\text{ iff }& \F,V,w\forces \varphi \text{ and } \F,V,w\forces \psi \\
\F,V,w&\forces \lnot\varphi &\text{ iff }& \F,V,w\notforces \varphi \\
\F,V,w&\forces \Box\varphi &\text{ iff }& \{x\in W: \F,V,x\forces\varphi\}\in N(w).
\end{align*}
The formula $\varphi$ is \emph{valid} on the frame $\F$, in symbols $\F\forces\varphi$, if for every evaluation $V$ and every world $w\in W$ we have $\F,V,w\forces\varphi$. For any neighborhood frame $\F$ the logic of the frame $\{\varphi: \F\forces\varphi\}$ is a congruential modal logic. 

The logic $\Lambda$ is \emph{neighborhood-complete} if there is a 
class $\mathsf{F}$ of neighborhood frames such that $\Lambda = \{\varphi: \F\forces\varphi$ for all $\F\in\mathsf{F}\}$. The logic $\Lambda$ is \emph{C-Post complete} if it is maximal consistent among congruential modal logics, that is, $\Lambda$ is congruential,  consistent, and has no proper consistent congruential extension.


\section{The solution}

We employ the algebraic method, which is briefly recalled here.
Given a class of similar algebras $\K$, the familiar class operators $\H(\K)$, $\S(\K)$, $\P(\K)$ will denote closure of $\K$ under taking homomorphic images, subalgebras and direct products respectively. A class of algebras $\K$ is a variety iff $\K=\H\S\P(\K)$. Due to Birkhoff's theorem, varieties are precisely the classes of algebras definable by equations. We will shortly write $\V(\K)$ for $\H\S\P(\K)$. Given an algebra $\A$, $\Con(\A)$ will denote the set of congruences of $\A$. In particular, $\Delta_\A$ and $\nabla_\A$ will denote the least and the largest elements of $\Con(\A)$, i.e. the identity congruence and the full congruence $A^2$, respectively. If $\Con(\A)=\{\Delta_\A,\nabla_\A\}$, then we say that $\A$ is simple. Given a subset $X\subseteq A$, $\Sg^\A(X)$ will denote the subalgebra of $\A$ generated by $X$. If $\A=\Sg^\A(X)$, we say that $X$ generates $\A$. We will say that an algebra $\A$ is $\kappa$ generated when $\A=\Sg^\A(X)$ and the cardinality of $X$ is $\kappa$. In particular, we say that $\A$ is zero-generated if $\A=\Sg^\A(\emptyset)$. Note that zero-generated algebras exist only if the similiarity type contains at least one constant symbol.
Given a class of algebras $\K$, the free algebra in $\S\P(\K)$ over a set of generators $X$ will be denoted by $\Fr_\K(X)$. A variety is trivial when it is generated by a one-element algebra (trivial algebra). A variety $\V$ will be called minimal when it is non-trivial and has only two subvarieties: itself and the trivial variety.  For the fundamentals of universal algebra, we refer the reader to \cite{BurrisSankappanavar1981}.

By a \emph{modal algebra} we understand a structure 
$\A = \<A, \land, -, \bot, \Box\>$, where $\<A, \land, -, \bot\>$ is a Boolean algebra,\footnote{We make use of other standard Boolean operations such as $\lor$, $\to$, $\top$, etc. These are defined in the usual way.}  and $\Box:A\to A$ is an arbitrary unary function.\footnote{In parts of the literature a modal algebra is a Boolean algebra with an extra \emph{operator} that is \emph{normal} and \emph{additive}, see e.g. \cite[Def. 10.1]{Ono2019}). In some other parts of the literature, e.g. in \cite{Krawczyk2023}, a modal algebra
is just a Boolean algebra with an arbitrary extra unary function. Yet in some other parts of the literature these structure are called Boolean frames, see \cite{Hansson1973,GyZMMOO-BSL,GyZMM-NDJFL}.} The class of modal algebras is an equational class (variety) defined by the Boolean equations. The congruentiality rule written in the quasi-equational form $x=y\rightarrow \Box x=\Box y$ holds automatically in any modal algebra ($\Box$ being a function). The logic $\E$ is Blok-Pigozzi algebraizable \cite{BlokPigozzi} w.r.t. the variety of modal algebras (cf. \cite{Krawczyk2023}). For this reason, the lattice of congruential extensions of $\E$ is dually isomorphic to the lattice of subvarieties of modal algebras. Therefore, in the algebraic setting, C-Post completeness for a given congruential modal logic amounts to minimality of the corresponding variety.

Let $\F=\<W,N\>$ be a neighborhood frame, and define $\F^{+} = \<\wp(W),\Box\>$, where $\Box X = \{w\in W: X\in N(w)\}$. Then $\F^{+}$ is a modal algebra, called the \emph{complex algebra} of $\F$, and validity of a formula $\varphi$ translates into validity of the equation $\varphi=\top$: 
\[
    \F\forces\varphi\qquad\text{ iff }\qquad \F^{+}\models \varphi=\top
\]
For any modal algebra $\A$, the set $\{\varphi: \A\models\varphi=\top\}$
is a congruential modal logic, and is called the logic of the algebra $\A$. Thus, if $\Lambda$ is the logic of the neighborhood frame 
$\F$, then 
\[
    \Lambda  = \{\varphi: \F\forces\varphi\} = \{\varphi:\F^{+}\models\varphi=\top\}.
\]
In particular, if the modal algebra $\A$ is isomorphic to some complex algebra $\F^{+}$ of a neighborhood frame $\F$, then the logic of the algebra $\A$ is neighborhood complete.
Finally, C-Post completeness of the logic of the algebra $\A$ translates to the property that $\A$ generates a non-trivial minimal variety. \\

Recall that an algebra is \emph{zero-generated} if it is generated as a subalgebra by the constant in the similarity type (in our case, $\bot$).
Below we generalise Theorem 3.1 in \cite{GyZMMOO-BSL} to an equivalence.

\begin{lemma}\label{zero-generated}
	Let $\V$ be a variety of modal algebras. Then $\V$ is a nontrivial minimal 	variety if and only if $\V=\V(\B)$ for some nontrivial, simple, zero-generated 	modal algebra $\B$.
\end{lemma}
\begin{proof}
	For the right to left direction let $\V = \V(\B)$ and consider the free  algebra $\Fr_{\V}(\emptyset)$ generated by the empty set in the variety $\mathbf{V}$ (it exists, because the similarity type contains nullary symbols). By the universal property of the free algebra there is a canonical homomorphism
	\[
		h: \Fr_{\V}(\emptyset) \to \B\,.
	\]
	The image is the subalgebra of $\B$ generated by the nullary symbols, which is the entire $\B$, hence $h$ is surjective. Since $\B$ generates the variety $\V$, if for two terms $t^{\B}=s^{\B}$, then this holds in the free algebra too: $t^{\Fr_{\V}(\emptyset)}=s^{\Fr_{\V}(\emptyset)}$, meaning that $h$ must be injective. Therefore, $h$ is an isomorphism $\Fr_{\V}(\emptyset)\cong  \B$. Let now $\mathbf{W}\subseteq\mathbf{V}$ be a subvariety. Then there is a canonical surjection 
	\[
		\Fr_{\V}(\emptyset) \twoheadrightarrow \Fr_{\mathbf{W}}(\emptyset)\,.
	\]
	But as $\Fr_{\V}(\emptyset)\cong  \B$, $\Fr_{\V}(\emptyset)$ is simple, and its homomorphic images are either trivial or itself. If $\mathbf{W}$ is nontrivial, then $\Fr_{\mathbf{W}}(\emptyset)$ is nontrivial, by properties of Boolean algebras, hence
	\[
		\Fr_{\mathbf{W}}(\emptyset)\cong \Fr_{\V}(\emptyset)\cong  \B\,.
	\]
	In this case $\B\in\mathbf{W}$ and therefore $\mathbf{W} = \V$, meaning that $\V$ has no proper nontrivial subvarieties.
	
	For the opposite direction suppose that $\V$ is nontrivial and minimal, and let $\B=\Fr_{\V}(\emptyset)$ be the zero-generated free algebra of $\V$. Since the similarity type of modal algebras contains the Boolean constants, this free algebra exists and is nontrivial. Thus $\V(\B)$ is a nontrivial subvariety of $\V$. By minimality of $\V$, it follows that $\V(\B)=\V$.
    It remains to prove that $\B$ is simple. Suppose, towards a contradiction, that $\B$ has a congruence $\Delta_{\B}\subsetneq\Theta\subsetneq\nabla_{\B}$. Put $\A =\B/\Theta$. Since $\Theta\neq\nabla_{\B}$, the algebra $\A$ is nontrivial.
    As $\B$ is zero-generated, every element of $\B$ is the value of some variable-free term. Since $\Theta\neq\Delta_{\B}$, there are variable-free terms $s$ and $t$ such that
    \[
        s^{\B}\neq t^{\B}     \qquad\text{and}\qquad
        (s^{\B},t^{\B})\in\Theta.
    \]
    Consequently, $\A\models s\approx t$. Hence $\V(\A)\models s\approx t$. On the other hand, $\B\not\models s\approx t$, and since $\B\in\V$, we have $\V\not\models s\approx t$. Therefore $\V(\A)\subsetneq\V$. But $\A$ is nontrivial, so $\V(\A)$ is a nontrivial proper subvariety of $\V$, contradicting the minimality of $\V$.
\end{proof}

The above lemma deserves special attention. Although it will appear crucial later in the proof, it also presents a serious obstacle to establishing the theorem we are after.
First, zero-generated algebras are always countable: they are homomorphic images of the free constant term algebra (which is countable for countable languages). Moreover, any algebra of the form $\F^+$ for some neighborhood frame is essentially a powerset algebra in the boolean reduct, so a countable algebra of this form is always finite. This implies that there are only countably many varieties of the form $\V(\F^+)$, where $\F^+$ is zero-generated. In order to establish the uncountability of the set of C-Post complete, neighborhood complete modal logics, we need to find a trick which allows us to go around this limitation. 

Consider $2^{\omega}$ as a topological space equipped with the product topology (the Cantor space). This topological space is compact, zero-dimensional (it has a clopen basis), complete metric space.

The following lemma is a special case of \cite[Theorem 3.1]{Ellis}, but for completeness we give a sketch of the proof.
\begin{lemma}\label{continuousextension}
    Let $D$ be a closed subset of $2^{\omega}$. 
    Every continuous map $g:D\to 2^{\omega}$ has a 
    continuous extension $G:2^{\omega}\to 2^{\omega}$.
\end{lemma}
\begin{proof}
    For every $n\in\omega$, let $C_n=\{x\in D:g(x)(n)=1\}$.
    Since the $n$th coordinate projection 
    $\pi_n:2^{\omega}\to 2$ is continuous, and $\{1\}$ is clopen
    in $2$, the set $C_n=(\pi_n\circ g)^{-1}(\{1\})$ is clopen in $D$.
    Every clopen subset of the closed subspace $D$ is of the form
    $D\cap U$ for some clopen $U\subseteq 2^{\omega}$ (because if $C$ is clopen in $D$, then $C$ and $D\setminus C$ are disjoint closed subsets of $2^{\omega}$, and the zero-dimensional compact space $2^{\omega}$ has a clopen set separating them). Thus, for
    each $n\in\omega$, one can choose a clopen set 
    $U_n\subseteq 2^{\omega}$ such that $D\cap U_n=C_n$.
    Define $G:2^{\omega}\to 2^{\omega}$ coordinatewise by
    \[
        G(x)(n)=
        \begin{cases}
            1&\text{ if } x\in U_n,\\
            0&\text{ if } x\notin U_n.
        \end{cases}
    \]
    For every $n$, the map $\pi_n\circ G:2^{\omega}\to 2$ is the
    characteristic function of the clopen set $U_n$, and hence is
    continuous. Since $2^{\omega}$ carries the product topology, it
    follows that $G$ is continuous. That $G$ extends $g$ is straightforward.
\end{proof}

\begin{theorem}\label{main result}
	There are exactly $2^{\aleph_0}$ congruential modal logics which are both C-Post complete and neighborhood complete. Each of the constructed logics is the logic of a single countable neighborhood frame.
\end{theorem}

\begin{proof}
	Write $\mathbf{0},\mathbf{1}\in 2^{\omega}$ for the 
    constant zero and constant one sequences, respectively. 
    For each $n\in\omega$, define the sequences 
    $a_n, c_n, u_n\in 2^{\omega}$ by
	\[
    	a_n(k)= \begin{cases}
	       1&\text{ if } k<n,\\
	       0&\text{ if } k\geq n,
	   \end{cases}
	   \qquad\qquad
    	c_n(k)= \begin{cases}
	       0&\text{ if } k=n,\\
	       1&\text{ if } k\neq n,
	   \end{cases}
	\]
	and
	\[
	   u_n(k)=\begin{cases}
	       1&\text{ if } k=n,\\
	       0&\text{ if } k\neq n.
	   \end{cases}
	\]
	Let $e\in 2^{\omega}$ be the sequence given by
	$e(k)=1$ iff $k$ is even, and, for each $n\in\omega$, let $e^{(n)}$
	be obtained from $e$ by changing its $n$th coordinate:
	\[
    	e_n(k)=\begin{cases}
	       1-e(n)&\text{ if } k=n,\\
	       e(k)&\text{ if } k\neq n.
	   \end{cases}
	\]
    Write
	\begin{align*}
	   D_0&=\{a_n: n\in\omega\}\cup\{\mathbf{1}\},\\
	   D_1&=\{c_n: n\in\omega\}\cup\{\mathbf{1}\},\\
	   D_2&=\{e\}\cup\{e_n: n\in\omega\},
	\end{align*}
	and $D=D_0\cup D_1\cup D_2$.
	In the product topology on $2^{\omega}$ the
    sequences $a_n$, $c_n$ and $e_n$ are convergent
    with $\lim a_n=\mathbf{1}$, $\lim c_n=\mathbf{1}$, 
    $\lim e_n= e$. Moreover, these limit points are the only accumulation points 
    of the corresponding sets $D_0$, $D_1$ and $D_2$. Hence $D_0$, $D_1$, $D_2$, and therefore $D$, are closed.

	Fix $S\subseteq\omega$ and define $g_S:D\to2^{\omega}$ by
	\begin{align}
    	g_S(a_n)&=a_{n+1}, & g_S(\mathbf{1})&=\mathbf{1},
            \label{eq:g-initial}\\
	   g_S(c_0)&=e, \label{eq:g-even}\\
	   g_S(c_{n+1})&= \begin{cases}
	       \mathbf{1}&\text{ if } n\in S,\\
	       c_{n+1} &\text{ if } n\notin S,
	       \end{cases}
	   & &n\in\omega, \label{eq:g-code}\\
	   g_S(e)&=\mathbf{0}, &
	   g_S(e_0)&=\mathbf{1},\label{eq:g-simple-zero}\\
	   g_S(e_{n+1})&=u_n, & &n\in\omega.
	   \label{eq:g-simple-step}
	\end{align}
	The only non-isolated points of $D$ are $\mathbf{1}$ and $e$. By
	\eqref{eq:g-initial}, 
    \[
        \lim g_S(a_n) = \lim a_{n+1} = \mathbf{1}.
    \]
    Also, $\lim g_S(c_n) = \mathbf{1}$, because apart from the first term, each $g_S(c_n)$ is either $\mathbf{1}$ or $c_n$, and $\lim c_n = \mathbf{1}$. Finally, 
	$\lim u_n = \mathbf{0} = g_S(e)$, so \eqref{eq:g-simple-step} shows that $\lim g_S(e_n) =  g_S(e)$. Thus $g_S$ is continuous. By 
    Lemma \ref{continuousextension}, $g_S$ has a continuous extension $G_S:2^{\omega}\to2^{\omega}$.

	Consider now $2^{\omega}$ as a Boolean algebra and put
	$\A_S = \< 2^{\omega},\land,\lor,\lnot, \mathbf{0},\mathbf{1}, G_S\>$.

    \begin{claim}\label{claim1}
        $\A_S$ is isomorphic to the complex algebra of a 
        neighborhood frame.
    \end{claim}
    \begin{proof}[of Claim \ref{claim1}]
    	We now construct a neighborhood frame whose complex algebra is isomorphic to $\A_S$. For $X\subseteq\omega$, let
	   $\chi_X\in2^{\omega}$ denote its characteristic function, and define
	   $\F_S=\<\omega, N_S\>$ by
	   \[
    	N_S(k)=	\{X\subseteq\omega: G_S(\chi_X)(k)=1\}.
	   \]
	   If $\Box_{S}$ denotes the modal operation of the complex algebra of $\F_S$, then for every $X\subseteq\omega$ and $k\in\omega$ 
       we have
	   \[
	       \chi_{\Box_{S}X}(k)=1
	           \quad\Longleftrightarrow\quad
	       X\in N_S(k)
	           \quad\Longleftrightarrow\quad
	       G_S(\chi_X)(k)=1.
	   \]
	   Hence
	   \begin{equation}\label{eq:complex-operation}
	       \chi_{\Box_{S}X} = G_S(\chi_X).
	   \end{equation}
	   Thus the map $X\mapsto\chi_X$ is an isomorphism from the complex modal algebra of $\F_S$ onto $\A_S$. 
    \end{proof}

	Let $\B_S$ be the zero-generated subalgebra of
	$\A_S$. As $a_0 = \mathbf{0}$, from \eqref{eq:g-initial} it
    follows that $G_S^n(\mathbf{0})=a_n$ for every $n\in\omega$. Therefore $a_{n+1}\wedge\neg a_n=u_n$ belongs to $B_S$ for every $n$. It follows
	that every sequence with finite support, and hence also every sequence
	whose zero-set is finite, belongs to $B_S$. In particular, $B_S$ is
	dense in $2^{\omega}$: every nonempty basic open set specifies only
	finitely many coordinates and therefore contains a sequence with finite
	support.

    \begin{claim}\label{claim2}
        The algebras $\A_S$ and $\B_S$ have the same
        equational theory. 
    \end{claim}
    \begin{proof}[of Claim \ref{claim2}]
    	One direction is straightforward because $\B_S$ is a 
        subalgebra of $\A_S$. For the converse, let
	    $\varphi(p_1,\ldots,p_m) = \top$ be valid in $\B_S$. 
        Its evaluation map
	   \[
	       \widehat\varphi:(2^{\omega})^m\to2^{\omega}
	   \]
	   is continuous, because the Boolean operations and $G_S$ are continuous. The singleton $\{\mathbf{1}\}$ is closed 
        in $2^{\omega}$, so $\widehat\varphi^{-1}(\{\mathbf{1}\})$ 
        is closed. Since $\varphi = \top$ is valid
	   in $\B_S$, the closed set $\widehat\varphi^{-1}(\{\mathbf{1}\})$
        contains the dense subset $B_S^m$.
        Consequently, 
        $\widehat\varphi^{-1}(\{\mathbf{1}\}) = (2^{\omega})^m$, yielding that the equation $\varphi=\top$ is valid in $\A_S$. 
    \end{proof}

    \begin{claim}\label{claim3}
        $\B_S$ is simple. 
    \end{claim}
    \begin{proof}[of Claim \ref{claim3}]
        Let $\theta$ be a non-identity congruence of $\B_S$, and let
	   \[
	       J=\{a\in B_S:a\mathrel\theta\mathbf{0}\}.
	   \]
	   The set $J$ is a nonempty Boolean ideal. 
       Choose a nonzero $a\in J$. There is some $n\in\omega$ with $a(n)=1$. Since $u_n\in B_S$ and $u_n\leq a$, we have $u_n\in J$.

	   The element $e$ belongs to $B_S$, because
       $e=G_S(c_0)$ by \eqref{eq:g-even}. 
       Let $\oplus$ denote the symmetric-difference operation
       $x\oplus y=(x\wedge\neg y)\vee(\neg x\wedge y)$.
       For every $n$ we have $e\oplus e_n = u_n$. Hence, 
       if $u_n\in J$, then $e\mathrel\theta e_n$.
       If $n>0$, applying $G_S$ and using
       \eqref{eq:g-simple-zero}--\eqref{eq:g-simple-step} 
       gives $\mathbf{0}\mathrel\theta u_{n-1}$.
       Iterating this argument yields $u_0\in J$. Hence
       $e\mathrel\theta e_0$, and another application of $G_S$, 
       together with \eqref{eq:g-simple-zero}, gives
       $\mathbf{0}\mathrel\theta\mathbf{1}$.
       Thus $\theta$ is the largest congruence, and therefore 
       $\B_S$ is simple.
    \end{proof}

    Let us write 
    \[
        \Lambda_S = \{ \varphi: \A_S\models \varphi=\top \}.
    \]
    Then $\Lambda_S$ is a congruential modal logic. By Claim \ref{claim1}, $\Lambda_S$ is complete with respect to the 
    single countable neighborhood frame $\F_S$. Since $\B_S$
    is zero-generated, and simple (by Claim \ref{claim3}), Lemma
    \ref{zero-generated} implies that $\V(\B_S)$ is a nontrivial, 
    minimal variety. This means that the logic corresponding
    to $\B_S$ is C-Post complete. As $\A_S$ and $\B_S$
    satisfy the same equations (by Claim \ref{claim2}), 
    $\Lambda_S$ is the logic corresponding to $\B_S$, and therefore
    $\Lambda_S$ is C-Post complete.
    To complete the proof of the present theorem, 
    it remains to show that the logics $\Lambda_S$ are pairwise distinct
    for distinct subsets $S\subseteq \omega$.

    \begin{claim}\label{claim4}
        For $S\neq T$ we have $\Lambda_S\neq\Lambda_T$.
    \end{claim}
    \begin{proof}[of Claim \ref{claim4}]
        Define the variable-free formulas
	    \[
	       \beta_n=\Box^n\bot,
	           \qquad
	       \delta_n=\beta_{n+1}\wedge\neg\beta_n.
	    \]
	    By \eqref{eq:g-initial}, their values in $\A_S$ are
        $\beta_n^{\A_S}=a_n$, and $\delta_n^{\A_S}=u_n$.
	    Put $\chi_n=\Box\neg\delta_{n+1}$. Since
	    $\neg u_{n+1}=c_{n+1}$, \eqref{eq:g-code} yields
        \[
	       \chi_n\in\Lambda_S \quad\Longleftrightarrow\quad	n\in S.
        \]
    	Thus $S\neq T$ implies $\Lambda_S\neq\Lambda_T$. 
    \end{proof}
Theorem \ref{main result} follows as a consequence of Claims \ref{claim1}--\ref{claim4}.
\end{proof}

\noindent Theorem \ref{main result} together with \cite[Theorem 4.1]{Krawczyk2026} answer two clusters of questions from \cite[p. 299]{Fritz}. 

\begin{quote}
{\it Two clusters of questions were already mentioned above: First, how many
modal logics Post complete in the set of congruential modal logics are the logic
of a class of neighborhood frames, and how many (if any) are not?}
\end{quote}
As it has been shown, there is an abundance of logics of each type: the answer to both of Fritz's questions is `continuum many'. 
\section*{Funding}
The research of the second author has been funded by NCN (Narodowe Centrum Nauki), grant no. 2023/49/N/HS1/04070.

\end{document}